\documentclass[onesided,11pt,letterpaper]{amsart}
\usepackage{datetime}
\usepackage{ragged2e}
\usepackage{amsmath}
\usepackage{booktabs} 
\usepackage[ruled]{algorithm2e} 

\newtheorem{theorem}{Theorem}
\newtheorem{lemma}{Lemma}

\newtheorem{proposition}{Proposition}

\newtheorem{corollary}{Corollary}

\usepackage{amsmath}
\usepackage{amssymb}
\usepackage{amsfonts}
\usepackage{amsthm}
\usepackage{mathtools}
\usepackage{bbm}
\mathtoolsset{%
}
\usepackage{relsize}

\usepackage[utf8]{inputenc}
\usepackage[T1]{fontenc}

\usepackage{dsfont}

\usepackage{mathrsfs}

\usepackage{accents}

\usepackage[dvipsnames,svgnames]{xcolor}
\colorlet{MyBlue}{DodgerBlue!60!Black}
\colorlet{MyGreen}{DarkGreen!85!Black}

\usepackage[margin=1in]{geometry}
\usepackage{fullpage}
\usepackage[font=small,labelfont=bf]{caption}
\usepackage{subfigure}
\usepackage{tikz}
\usetikzlibrary{calc,patterns}
\usetikzlibrary{arrows,shapes,positioning}

\usepackage{acronym}
\usepackage{booktabs}       
\usepackage{latexsym}
\usepackage{paralist}
\usepackage{wasysym}
\usepackage{xspace}
\usepackage{comment}
\usepackage{longtable}
\usepackage[shortlabels]{enumitem}

\usepackage[authoryear,comma,numbers]{natbib}

\usepackage{hyperref}
\hypersetup{
colorlinks=true,
linktocpage=true,
pdfstartview=FitH,
breaklinks=true,
pdfpagemode=UseNone,
pageanchor=true,
pdfpagemode=UseOutlines,
plainpages=false,
bookmarksnumbered,
bookmarksopen=false,
bookmarksopenlevel=1,
hypertexnames=true,
pdfhighlight=/O,
urlcolor=MyBlue!60!black,linkcolor=MyBlue!70!black,citecolor=DarkGreen!70!black, 
pdftitle={},
pdfauthor={},
pdfsubject={},
pdfkeywords={},
pdfcreator={pdfLaTeX},
pdfproducer={LaTeX with hyperref}
}

\numberwithin{equation}{section}  
\usepackage[sort&compress,capitalize,nameinlink]{cleveref}
\crefname{app}{Appendix}{Appendices}

\crefrangeformat{equation}{\upshape(#3#1#4)\textendas\testh(#5#2#6)}

\newcommand{\R}{\mathbb{R}}

\newcommand{\N}{\mathbb{N}}

\DeclareMathOperator{\expo}{e}

\DeclareMathOperator{\Var}{Var}

\DeclarePairedDelimiter{\braces}{\{}{\}}
\DeclarePairedDelimiter{\bracks}{[}{]}
\DeclarePairedDelimiter{\parens}{(}{)}

\DeclarePairedDelimiter{\abs}{\lvert}{\rvert}

\DeclarePairedDelimiterX{\braket}[2]{\langle}{\rangle}{#1,#2}

\DeclarePairedDelimiterX{\inner}[2]{\langle}{\rangle}{#1,#2}
\DeclarePairedDelimiterX{\setdef}[2]{\{}{\}}{#1:#2}

\DeclarePairedDelimiterXPP{\probof}[1]{\Prob}{(}{)}{}{%
	
	#1}

\DeclarePairedDelimiterXPP{\exof}[1]{\Expect}{[}{]}{}{%
	
	#1}

\usepackage[textwidth=30mm]{todonotes}

\newcommand{\debug}[1]{#1}

\newcommand{\Expect}{\mathsf{\debug E}}
\newcommand{\Prob}{\mathsf{\debug P}}

\newcommand{\proba}{\debug p}
\newcommand{\modproba}{\widetilde{\proba}}
\newcommand{\distr}{\debug F}
\newcommand{\moddistr}{\widetilde{\distr}}

\newcommand{\mgf}{\debug \phi}

\newcommand{\ldrate}{\debug I}
\newcommand{\modldrate}{\widetilde{\ldrate}}
\newcommand{\salgebra}{\mathcal{\debug F}}

\DeclareMathOperator{\Poisson}{\mathsf{\debug{Poisson}}}

\newcommand{\game}{\debug \Gamma}

\newcommand{\play}{\debug i}

\newcommand{\nplayers}{\debug n}
\newcommand{\players}{\bracks{\nplayers}}

\newcommand{\act}{\debug s}

\newcommand{\actprof}{\boldsymbol{\act}}
\newcommand{\actprofalt}{\actprof'}
\newcommand{\actions}{\debug \Sigma}

\newcommand{\pay}{\debug u}

\DeclareMathOperator{\ASU}{\mathsf{ \debug {ASU}}}

\DeclareMathOperator{\BEq}{\mathsf{\debug {Beq}}}

\DeclareMathOperator{\typ}{\mathsf{typ}}
\DeclareMathOperator{\opt}{\mathsf{opt}}
\DeclareMathOperator{\beq}{\mathsf{beq}}
\DeclareMathOperator{\weq}{\mathsf{weq}}
\DeclareMathOperator{\NE}{\mathsf{\debug {NE}}}

\DeclareMathOperator{\PoA}{\mathsf{\debug {PoA}}}
\DeclareMathOperator{\SO}{\mathsf{\debug {SO}}}

\DeclareMathOperator{\PoS}{\mathsf{\debug {PoS}}}

\DeclareMathOperator{\WEq}{\mathsf{\debug {Weq}}}

\newcommand{\diff}{\, \textup{d}}

\newcommand{\ind}{\mathbf{1}}

\newcommand{\cG}{\ensuremath{\mathcal G}}

\newcommand{\cN}{\ensuremath{\mathcal N}}

\newacro{NE}{Nash equilibrium}
\newacroplural{NE}[NE]{Nash equilibria}
\newacro{PNE}{pure Nash equilibrium}
\newacroplural{PNE}[PNE]{pure Nash equilibria}
\newacro{MNE}{mixed Nash equilibrium}
\newacroplural{MNE}[MNE]{mixed Nash equilibria}
\newacro{PFNE}{prior-free Nash equilibrium}
\newacroplural{PFNE}[PFNE]{prior-free Nash equilibria}
\newacro{WE}{Wardrop equilibrium}
\newacroplural{WE}[WE]{Wardrop equilibria}
\newacro{SO}{socially optimum}
\newacro{SU}{social utility}
\newacro{ASU}{average social utility}
\newacroplural{ASU}{average social utilities}
\newacro{BEq}{best equilibrium}
\newacro{WEq}{worst equilibrium}
\newacro{KKT}{Karush\textendash Kuhn\textendash Tucker}
\newacro{OD}[O/D]{origin-destination}
\newacro{PoA}{price of anarchy}
\newacro{PoS}{price of stability}
\newacro{PoCS}{price of correlated stability}
\newacro{BPR}{bureau of public roads}
\newacro{FIP}{finite improvement property}
\newacro{CLT}{central limit theorem}

\newacro{BPG}{buck-passing game}
\newacro{SBPG}{stochastic buck-passing game}
\newacro{MBPG}{mixed extension of the buck-passing game}

\title[Efficiency of equilibria in games with random payoffs]{Efficiency of equilibria in games with random payoffs}

\author[M.~Quattropani]{Matteo Quattropani$^{\dagger}$}
\address{$^{\dagger}$ Dipartimento di Economia e Finanza, LUISS, Viale Romania 32, 00197 Roma, Italy.}
\email{mquattropani@luiss.it}
\author[M.~Scarsini]{Marco Scarsini$^{\#}$}
\address{$^{\#}$ Dipartimento di Economia e Finanza, LUISS, Viale Romania 32, 00197 Roma, Italy.}
\email{marco.scarsini@luiss.it}

\begin{document}
	
	\maketitle

\begin{abstract}
	We consider normal-form games with $\nplayers$ players and two strategies for each player, where the payoffs are i.i.d. random variables with some distribution $\distr$ and we consider issues related to the pure equilibria in the game as the number of players diverges. 
	It is well-known that, if the distribution $\distr$ has no atoms, the random number of pure equilibria is asymptotically $\Poisson(1)$. 
	In the presence of atoms, it diverges. 
	For each strategy profile, we consider the (random) average payoff of the players, called \acl{ASU}.
	In particular, we examine the asymptotic behavior of the optimum \acl{ASU} and the one associated to the best and worst \acl{PNE} and we show that, although these quantities are random, they converge, as $\nplayers\to\infty$ to some deterministic quantities.	
\end{abstract}

	\section{Introduction}
	\label{se:intro}
	%
	%

	The concept of \ac{NE} is central in game theory. 
	\citet{Nash:PNAS1950,Nash:AM1951} proved that every finite game admits \acp{MNE}, but, in general, \acp{PNE} may fail to exist.
	The concept of \acl{PNE} is epistemically better understood than the one of \ac{MNE}, so it is important to understand how rare are games without \acp{PNE}. 
	One way to address the problem is to consider games with random  payoffs.
	In  a random game the number of \acp{PNE} is also a random variable, whose distribution  depends on the assumptions about the distribution of the random payoffs. 
	The simplest case that has been considered in the literature deals with i.i.d.\ payoffs having a continuous distribution function.
	This implies that ties among payoffs have zero probability. 
	Even in this simple case, although it is easy to compute the expected number of \acp{PNE}, the characterization of their exact distribution is non-trivial. 
	Asymptotic results exist as either the number of players or the number of strategies for each player diverges. 
	In both cases the number of \acp{PNE} converges to a Poisson distribution with parameter $1$.
	Generalizations of the simple case have been considered, for instance either by removing the assumptions that all payoffs are independent or by allowing for discontinuities in their distribution functions. 
	In both cases the number of \acp{PNE} diverges and some \ac{CLT} holds. 
	
	The literature on games with random payoffs has focused on the distribution of the number of \acp{PNE}. 
	To the best of our knowledge, the distribution of their \ac{ASU}, i.e., the average payoff of each player has never been studied.
	
	The (in)efficiency of equilibria is a central topic in algorithmic game theory.
	In the case of positive payoffs, one way to measure this inefficiency is through the \ac{PoA}, i.e., the ratio of the optimal \ac{ASU} over the \ac{ASU} of the worst \ac{NE}, or through the \ac{PoS}, i.e., the ratio of the optimal \ac{ASU} over the \ac{ASU} of the best \ac{NE}.
	To analyze the efficiency of equilibria in games with random payoffs, it is then important to study the behavior of the random optimal \ac{ASU} and the \ac{ASU} of the best and worst \ac{PNE}.
	
	%
	
	\subsection{Our contribution}
	
	The goal of this paper is to study the asymptotic behavior of the \ac{ASU} of the optimum and of the \ac{PNE} of games with random payoffs when the number of players increases. 
	We consider normal-form games with $\nplayers$ players and two strategies for each player, where payoffs are assumed to be i.i.d.\ random variables having common distribution $\distr$. 
	We first show that the optimal \ac{ASU} converges in probability to a deterministic value that can be characterized in terms of the large deviation rate of $\distr$.
	
	Then we move to examine the asymptotic \ac{ASU} of the \acp{PNE}.
	We start by considering the case in which $\distr$ has no atoms. In this case, as shown in \citet{RinSca:GEB2000}, asymptotically the number of \ac{PNE} has a Poisson distribution with mean $1$.
	This implies that we typically do not have many equilibria.
	We will show that, when equilibria exist, in the limit they all have the same \ac{ASU}.
	
	We then consider the case in which $\distr$ has some atoms, so that the number of equilibria grows exponentially in probability, as show in \citet{AmiColSca:arXiv2019}. 
	We will show that asymptotically all but a vanishingly small fraction of equilibria share the same \ac{ASU}. 
	On the other hand we show that the \ac{ASU} of the best and the worst equilibrium, converge in probability to two values that depend on $\distr$.
	These values can be characterized by means of a large deviation estimate of a modified distribution $\moddistr$ that depends on $\distr$ in an explicit way. 
	In particular, given a random variable $X$ with distribution $\distr$, $\moddistr$ is the distribution function of $X$ conditionally on $X$ being larger or equal than an independent copy $X'$.
	So, even if most \ac{PNE} have the same asymptotic \ac{ASU}, the \ac{ASU} of the worst and best \ac{PNE} can be quite different. 
	
	Finally, we consider the special case where $\distr$ is a Bernoulli distribution with parameter $\proba$.
	We analyze the limit behavior of the \ac{ASU} of the optimum, the best, and worst \ac{PNE}, as a function of $\proba$ and study these functions.

	\subsection{Related literature}
	
	The distribution of the number of \ac{PNE} in games with random payoffs has been studied for a number of years. 
	Many papers assume the random payoffs to be i.i.d.\ from a continuous distribution and study the asymptotic behavior of random games, as the number of strategies grows. 
	For instance, \citet{Gol:AMM1957}  showed that in zero-sum two-person games  the probability of having a \ac{PNE} goes to zero.
	He also briefly dealt with the case of payoffs having a Bernoulli distribution.
	\citet{GolGolNew:JRNBSB1968} studied general two-person games and showed that the probability of having at least one \ac{PNE} converges to $1-\expo^{-1}$.
	\citet{Dre:JCT1970} generalized this result to the case of an arbitrary finite number of players.
	Other papers have looked at the asymptotic distribution of the number of \ac{PNE}, again when the number of strategies diverges.
	\citet{Pow:IJGT1990} showed that, when the number of strategies of at least two players goes to infinity, the distribution of the number of \ac{PNE} converges to a $\Poisson(1)$.
	She then compared the case of continuous and discontinuous distributions. 
	\citet{Sta:GEB1995} derived an exact formula for the distribution of \ac{PNE} in random games and obtained the result in \citet{Pow:IJGT1990} as a corollary.
	\citet{Sta:MOR1996} dealt with the case of two-person symmetric games and obtained Poisson convergence for the number of both symmetric and asymmetric \ac{PNE}.
	
	In all the games with a continuous distribution of payoffs, the expected number of \ac{PNE} is in fact $1$ for any fixed $\nplayers$.
	Under different hypotheses, this expected number diverges. 
	For instance, \citet{Sta:MSS1997,Sta:EL1999} showed that this is the case for games with vector payoffs and for games of common interest, respectively.
	\citet{RinSca:GEB2000} weakened the hypothesis of i.i.d.\ payoffs; that is, they assumed that  payoff vectors corresponding to different strategy profiles are i.i.d., but they allowed some dependence within the same payoff vector. 
	In this setting, they proved asymptotic results when either the number of players or the number of strategies diverges. 
	More precisely, if each payoff vector has a multinormal exchangeable distribution with correlation coefficient $\rho$, then, if $\rho$ is positive,  the number of \ac{PNE} diverges and a central limit theorem holds.
	\citet{Rai:P7YSM2003} used Chen-Stein method to bound the distance between the distribution of the normalized number of \acp{PNE} and a normal distribution. His result is very general, since it does not assume continuity of the payoff distributions. 
	\citet{Tak:GEB2008} considered the distribution of the number of \ac{PNE} in a random game with two players, conditionally  on the game having nondecreasing best-response functions.
	This assumption greatly increases the expected number of \ac{PNE}.
	\citet{DasDimMos:AAP2011} extended the framework of games with random payoffs to graphical games. Players are vertices of a graph and their strategies are binary, like in our model. Moreover, their payoff depends only on their strategy and the strategies of their neighbors.
	The authors studied how the structure of the graph affects existence of \ac{PNE} and they examined both deterministic and random graphs.
	\citet{AmiColSca:arXiv2019} showed that in games with $\nplayers$ players and two actions for each player, the key quantity that determines the behavior of the number of \acp{PNE} is the probability that two different payoffs assume the same value. 
	They then studied the behavior of best-response dynamics in random games.
	\citet{AmiColHam:ORL2021} compared the asymptotic behavior of best-response and better-response dynamics in two-person games with random payoffs with a continuous distribution, as the number of strategies diverges. 
	Properties of learning procedures in games with random payoffs have been studied by \citet{DurGau:AGT2016,GalFar:PNAS2013,PanHeiFar:SA2019,HeiJanMunPanScoTarWie:arxiv2021}.
	
	The issue of solution concepts in games with random payoffs has been explored by various authors in different directions.
	For instance, \citet{Coh:PNAS1998} studied the probability that Nash equilibria (both pure and mixed) in a finite random game maximize the sum of the players' payoffs.
	This bears some relation with what we do in this paper.
	
	The fact that selfish behavior of agents produces inefficiencies goes back at least to \citet{Pig:Macmillan1920} and has been studied in various fashions in the economic literature.
	Measuring inefficiency of equilibria in games has attracted the interest of the algorithmic-game-theory community around the change of the millennium.
	Efficiency of equilibria is typically measured using either the \ac{PoA} or the \ac{PoS}. 
	The \ac{PoA}, i.e., the ratio of the optimum \ac{SU} over the \ac{SU} of the worst equilibrium, was introduced by \citet{KouPap:STACS1999} and given this name by  \citet{Pap:ACMSTC2001}.
	The \ac{PoS}, i.e., the ratio of the optimum \ac{SU} over the \ac{SU} of the best equilibrium, was introduced by \citet{SchSti:P14SIAM2003} and given this name by  \citet{AnsDasKleTarWexRou:SIAMJC2008}.
	The reader is referred for instance to \citet{RouTar:inCUPress2007} for the basic concepts related to inefficiency of equilibria.

	\subsection{Organization of the paper}\label{suse:organization}
	\cref{sec:model} introduces the model and the main results. 
	\cref{sec:SO} shows the convergence in probability of the \ac{SO}, regardless the existence of atoms in the distribution of the payoffs. 
	In \cref{sec:first-mom} we provide an operative expression for the expected number of equilibria having \ac{ASU} exceeding a given threshold.
	\cref{sec:noties,sec:ties} are devoted to the proof of the main results for the model with no ties and with ties, respectively. 
	Finally, in \cref{sec:binary}, we study a specific instance of the model, in which the distribution $\distr$ is Bernoulli, providing more explicit results.

	%
	%
	\section{Model and main results}
	\label{sec:model}
	
	Throughout the paper we adopt the usual asymptotic notation. More precisely, for any two real sequences $(f_\nplayers)_{\nplayers\in\N}$ and $(g_\nplayers)_{\nplayers\in\N}$
	\begin{equation*}
	f_\nplayers=O(g_\nplayers)\iff\lim_{\nplayers\to\infty}\frac{f_\nplayers}{g_\nplayers}<\infty,\quad\text{and}\quad f_\nplayers=o(g_\nplayers)\iff\lim_{\nplayers\to\infty}\frac{f_\nplayers}{g_\nplayers}=0.
	\end{equation*}
	For a set $A$, the symbol $\abs{A}$ will denote its cardinality.	In the background we will always have a probability space $\parens*{\Omega,\salgebra, \Prob}$ on which all the random quantities are defined. 
	In particular, we will consider a sequence $\parens*{\game_{\nplayers}}_{\nplayers\in\N}$ of normal form games where $\game_{\nplayers}$ has  $\nplayers$ players.
	The set of players of game $\game_{\nplayers}$ is denoted by $\players$.
	Each player can choose one strategy in $\braces{0,1}$; then, the set $\actions_{\nplayers}$  of strategy profiles is the Cartesian product $\braces{0,1}^{\nplayers}$. As in \citet{DasDimMos:AAP2011}, the symbol $\oplus$ will denote the binary XOR operator, defined as 
	\begin{equation}
	\label{eq:XOR}
	1 \oplus 0=0 \oplus 1 = 1,\quad 0\oplus 0 = 1 \oplus 1 = 0.
	\end{equation}
	Therefore, $\oplus$-adding $1$ changes one strategy into the other. 
	
	For $\play\in[\nplayers]$, $\pay_{\play} \colon \actions_{\nplayers} \to \R$ denotes player $\play$'s payoff function. 
	We further assume that all payoffs of all games are i.i.d. random variables with the same marginal distribution $\distr$, i.e.,
	\begin{equation}
	\label{eq:binary-payoff}
	\Prob(\pay_\play(\actprof)\le x)=\distr(x),\quad\forall\nplayers\in\N,\:\forall\play\in\players,\:\forall \actprof\in\actions_{\nplayers},\:\forall x\in\R. 
	\end{equation}
	A strategy profile $\actprof\in\actions_{\nplayers}$ is a \acl{PNE} if
	\begin{equation*}
	\forall \play\in[\nplayers],\qquad u_{\play}(\actprof)\ge u_{\play}((\actprof_{-\play},\act_{\play}\oplus 1)),
	\end{equation*}
	where $\actprof_{-\play}$ is the subprofiles of the strategies of all players except $\play$.
	Let $\NE(\game_{\nplayers})$ denote the set of \aclp{PNE}.
	We will be interested in the asymptotic behavior, as $\nplayers\to\infty$, of the following quantities:
	\begin{align}
	&\label{eq:SU}
	\text{\acfi{ASU}\acused{ASU} } &  \ASU(\game_{\nplayers},\actprof) &\coloneqq \frac{1}{\nplayers}\sum_{\play\in[\nplayers]}\pay_{\play}(\actprof),\\
	&\text{\acfi{SO}\acused{SO} } & \SO(\game_{\nplayers})&\coloneqq \max_{\actprof\in\actions}\ASU(\actprof),\\
	&\text{\acfi{BEq}\acused{BEq} } & \BEq(\game_{\nplayers}) &\coloneqq \max_{\actprof\in\NE}\ASU(\actprof),\\
	&\text{\acfi{WEq}\acused{WEq} } & \WEq(\game_{\nplayers}) &\coloneqq \min_{\actprof\in\NE}\ASU(\actprof).
	\end{align}
	We say that a sequence of real random variables $(U_{\nplayers})_{\nplayers\in\N}$  \emph{converges in probability to} $c\in\R$  (denoted by $U_{\nplayers}\overset{\Prob}{\longrightarrow}c$), if
	\begin{equation}
	\forall\varepsilon>0,\quad\lim_{\nplayers\to\infty}\Prob\big(\big|U_{\nplayers}-c\big|<\varepsilon \big)=1.
	\end{equation}
	Given the payoff distribution $\distr$, and two independent random variables $X,X'\sim \distr$, let
	\begin{equation*}
	\alpha\coloneqq\Prob(X=X'),\quad \beta\coloneqq\frac{1-\alpha}{2}=\Prob(X>X').
	\end{equation*}
	Notice that if $\distr$ is continuous then $\alpha=0$ and $\beta=1/2$. 
	In general, $\distr$ can be decomposed into a continuous and an atomic part. 
	Call $L$ the set of atoms of $\distr$. 
	Then
	\begin{equation}
	\label{eq:alphas}
	\alpha=\sum_{\ell\in L}\alpha_\ell, \quad\text{with}\quad\alpha_\ell\coloneqq\Prob(X=\ell)^2.
	\end{equation}
	%
	For the sake of simplicity, we assume the existence of a density function $f:\R\to\R_+$ such that 
	\begin{equation*}
	\distr(x)=\parens*{1-\sum_{\ell\in L}\sqrt{\alpha_\ell}}\int_{-\infty}^{x}f(y)dy+\sum_{\ell\in L}\sqrt{\alpha_\ell}\ind_{\ell\le x}.
	\end{equation*}
	We will further assume that the distribution $\distr$ has exponential moments of all order, i.e., if $X\sim \distr$, 
	\begin{equation}\label{eq:hp-mgf}
	\mgf(t)\coloneqq\Expect[e^{tX}]<\infty,\quad\forall t\in\R.
	\end{equation}
	In other words, the moment generating function of $\distr$ is everywhere finite.
	Notice that, for fixed $\nplayers$, the collection of  random variables $(\ASU(\game_{\nplayers},\actprof))_{\actprof\in\actions_{\nplayers}}$ is  i.i.d. and 
	the random variable $\SO(\game_{\nplayers})$  is the maximum of $2^{\nplayers}$ independent random variables with common distribution. 
	Therefore, its behavior can be analyzed using classical tools in extreme value theory.
	In fact, for all possible $\alpha$, our analysis of $\SO(\game_{\nplayers})$ relies on the study of the large deviations of the random variable
	\begin{equation*}
	\frac{1}{\nplayers}\sum_{\play=1}^{\nplayers} X_{\play},\quad X_{\play}\sim \distr,\:\text{i.i.d.}.
	\end{equation*} 
	If we define the large deviation rate
	\begin{equation}
	\label{eq:def-rate}
	\ldrate(x)\coloneqq\sup_{t\in\R}\bracks*{xt-\log\big(\phi(t) \big) },
	\end{equation}			
	then, under the assumption in \cref{eq:hp-mgf}, from Cramer's large deviation theorem (see, e.g., \citet[Theorem~1.4]{den:AMS2000}), it follows that, 
	\begin{equation*}
	\forall x> \Expect[X],\quad \lim_{\nplayers\to\infty}\frac{1}{\nplayers}
	\log\parens*{\Prob\parens*{\frac{1}{\nplayers}\sum_{\play=1}^{\nplayers} X_\play\ge x}}=-\ldrate(x).
	\end{equation*}
	Moreover, the function $\ldrate$ is lower semi-continuous and convex on $\R$, and
	\begin{equation*}
	\ldrate(x)\ge0\quad\text{and}\quad \ldrate(x)=0\iff x=\Expect[X],
	\end{equation*}
	\citep[see][Lemma 1.14]{den:AMS2000}.
	
	The following proposition describes the asymptotic behavior of the \ac{SO}. 
	Notice that \cref{pr:SO} does not require any assumption about $\alpha$, i.e., the asymptotic  behavior of the \ac{SO} does not depend on the presence of atoms in $\distr$.
	\begin{proposition}
		\label{pr:SO}
		Let $\distr$ satisfy \cref{eq:hp-mgf} and let $\ldrate$  be its large deviation rate. Then
		\begin{equation}
		\SO(\game_{\nplayers})\overset{\Prob}{\longrightarrow} x_{\opt},
		\end{equation}
		where
		\begin{equation}\label{eq:def-xopt}
		x_{\opt}\coloneqq\inf\left\{x>\Expect[X]\:\big\rvert\: \ldrate(x)> \log(2)  \right\}.
		\end{equation}			
	\end{proposition}
	
	The intuition underlying \cref{pr:SO} is that a deviation from the typical \ac{ASU}, i.e., $\Expect[X]$, is exponentially rare. In particular, given $ x>\Expect[X]$, the probability to have an \ac{ASU} larger or equal than $x$ is roughly $\exp(-\ldrate(x)\cdot\nplayers)$. Since the number of strategy profiles is $2^{\nplayers}$ as soon as $\ldrate(x)>\log(2)$ the expected number of strategy profiles with \ac{ASU} larger than $x$ tends to zero.
	
	We now consider the asymptotic behavior of the best and worst  \ac{PNE}.
	This analysis is more complicated than what we had for the \ac{SO}. In fact, the procedure requires an optimization of the random set of pure Nash equilibria. 
	We will distinguish between the case where the payoffs have no ties ($\alpha=0$)  and when they can have ties ($\alpha>0$). 
	\citet{RinSca:GEB2000} showed that, when $\alpha=0$,
	\begin{equation}
	\forall k\in\N, \quad\lim_{\nplayers\to\infty}\Prob(\abs*{\NE(\game_{\nplayers})}=k)=\frac{e^{-1}}{k!}.
	\end{equation}
	That is, if there are no ties, the number of pure equilibria converges weakly to a Poisson distribution of parameter $1$. 
	This implies that typically the equilibria are not numerous and, with positive probability, the set of equilibria can even be empty. 
	In this scenario, we use a first moment argument to show that, if pure equilibria exist, then  asymptotically they all share the same \ac{ASU}. 
	
	The results about the asymptotic behavior of \ac{BEq} and \ac{WEq} will require the following definition.
	For $X,X'\sim\distr$, let
	\begin{equation}
	\label{eq:def-Ftilde}
	\moddistr(y)\coloneqq
	\Prob\parens*{X \le y \mid X\ge X'}
	=\frac{2}{1+\alpha}\bracks*{\distr^2(y)-\int_{(-\infty,y]}\distr(x^-)\diff \distr(x)}.
	\end{equation}
	Notice that, when $\alpha=0$, we have 
	\begin{equation}\label{eq:def-ftildenoties}
	\moddistr(y)= \distr(y)^2.
	\end{equation}
	
	The following theorem deals with the asymptotic behavior of the class of  \acp{PNE}.
	\begin{theorem}[Convergence (no ties)]
		\label{th:convergence-noties}
		Let $\distr$ be such that $\alpha=0$ and let
		\begin{equation}
		\label{eq:def-xtyp-noties}
		x_{\typ}=\Expect[Y],\quad\text{with}\quad Y\sim\moddistr.
		\end{equation}
		Then
		\begin{equation}\label{eq:result-noties}
		\ind_{\NE(\game_{\nplayers})\neq\varnothing}\max_{\actprof\in\NE(\game_{\nplayers}) } \abs*{\ASU(\game_{\nplayers},\actprof)-x_{\typ}}\overset{\Prob}{\longrightarrow}0.
		\end{equation}
	\end{theorem}
	
	Conversely, when $\alpha>0$, we know by \citet{AmiColSca:arXiv2019} that the number of pure Nash equilibria is exponentially large in the number of players. 
	In particular,
	\begin{equation}\label{eq:amiet1}
	\frac{1}{\nplayers}\log\abs{\NE(\game_{\nplayers})}\overset{\Prob}{\longrightarrow}\log\parens*{1+\alpha}.
	\end{equation}
	
	We cannot analyze the asymptotic behavior of \ac{BEq} and \ac{WEq} along the lines of \cref{pr:SO} because of the stochastic dependence of the payoffs corresponding to different \acp{PNE}.
	To be more precise, if $\actprof$ is a \ac{PNE}, and $\actprofalt$ is a neighbor of $\actprof$, i.e., the two profiles differ only in one coordinate, say $\play$, then $u_{\play}(\actprofalt)$ is not independent of $u_{\play}(\actprof)$.
	The proof of the following theorem it will be based on the fact that this dependence, although present, is weak.
	
	\begin{theorem}[Convergence (with ties)]\label{th:convergence-ties}
		Let $\distr$ be such that $\alpha>0$.
		Let $x_{\typ}$ be defined as in \cref{eq:def-xtyp-noties} and let $\modldrate$ be the  large deviation rate associated to $\moddistr$. 
		\begin{enumerate}[\upshape(i)]
			\item 
			\label{it:th:convergence-ties-1}
			If
			\begin{align}\label{eq:def-xbeq}
			&x_{\beq}\coloneqq\inf\left\{x>x_{\typ}\::\: \modldrate(x)>\log(1+\alpha) \right\},\\
			\label{eq:def-xweq}&x_{\weq}\coloneqq\sup\left\{x<x_{\typ}\::\: \modldrate(x)<\log(1+\alpha) \right\},
			\end{align}
			then
			\begin{equation}\label{eq:result-ties}
			\big(\BEq(\game_{\nplayers}),\WEq(\game_{\nplayers}) \big)\overset{\Prob}{\longrightarrow}\big(x_{\beq},x_{\weq} \big).
			\end{equation}
			
			\item 
			\label{it:th:convergence-ties-2}			
			If
			\begin{equation}\label{eq:thm3-2}
			\NE_{\typ,\varepsilon}(\game_{\nplayers})\coloneqq\left\{\actprof\in\NE(\game_{\nplayers})\::\:\big|\ASU(\game_{\nplayers},\actprof)-x_{\typ} \big| < \varepsilon\right\},
			\end{equation}
			then
			\begin{equation}\label{eq:result-typ}
			\forall\varepsilon>0,\quad\frac{ |\NE_{\typ,\varepsilon}(\game_{\nplayers})|}{ |\NE(\game_{\nplayers})|}\overset{\Prob}{\longrightarrow}1.
			\end{equation}
		\end{enumerate}
	\end{theorem}
	In other words, \cref{th:convergence-ties} states that  most of the equilibria share the ``typical'' efficiency but, since they are exponentially many, some of them with a macroscopically larger \ac{ASU}. 
	Moreover, the efficiency of the best and worst equilibria do not fluctuate but converge to the solution of an explicit optimization problem.

	Since we are interested in the asymptotic properties of a random game $\game_{\nplayers}$ when $\nplayers\to\infty$, we usually neglect the dependence on $\nplayers$ (and on $\game_{\nplayers}$) when it is clear from the context. 
	
	\section{The Social Optimum}\label{sec:SO}
	In this section we prove the convergence in probability of the \ac{SO}, regardless of the value of $\alpha$. 
	This result is an immediate consequence of Cramer's large deviation theorem (see \citet{den:AMS2000}). 
	Our proof will require the definition of the following sets:
	\begin{equation}\label{eq:def-W}
	W^+_x(\game_{\nplayers})=\left\{\actprof\in\actions_{\nplayers}\::\: \ASU(\game_{\nplayers},\actprof)\ge x \right\}\quad\text{and}\quad W^-_x(\game_{\nplayers})=\left\{\actprof\in\actions_{\nplayers}\::\: \ASU(\game_{\nplayers},\actprof)\le x \right\}.
	\end{equation}
	In words, $W_x^+$ is the set of strategy profiles with an \ac{ASU} at least $x$ and, similarly, $W_x^+$ is the set of strategy profiles with an \ac{ASU} at most $x$.
	
	\begin{proof}[Proof of \cref{pr:SO}]
		We start by noticing that the claim in \cref{pr:SO} is equivalent to 	\begin{align}
		\label{eq:claim-SO1}
		\forall\varepsilon>0,\quad &\Prob\parens*{|W^+_{x_{\opt}+\varepsilon}|=\varnothing }\to 1,
		\shortintertext{and}
		\label{eq:claim-SO2}
		\forall\varepsilon>0,\quad &\Prob\parens*{|W^+_{x_{\opt}-\varepsilon}|=\varnothing }\to 0.
		\end{align}
		Recall that the \ac{ASU} of any given strategy profile has law
		\begin{equation*}
		\Prob\parens*{\ASU(\actprof)\le x } =\Prob\parens*{\frac{1}{\nplayers}\sum_{i=1}^{\nplayers} X_{\play}\le x},
		\end{equation*}
		where $(X_\play)_{\play\le \nplayers}$ is a collection of i.i.d. random variables with law $\distr$. Hence, by Cramer's large deviation theorem (see \citet{den:AMS2000}), 
		\begin{equation}\label{eq:ldp}
		\forall x> \Expect[X],\quad \lim_{\nplayers\to\infty}\frac{1}{\nplayers}\log\parens*{\Prob\parens*{\frac{1}{\nplayers}\sum_{\play=1}^{\nplayers} X_{\play}\ge x } }=-\ldrate(x).
		\end{equation}
		Therefore, by \cref{eq:ldp} the expected size of $W_x^+$ is given by
		\begin{equation}\label{eq:1stmom-ldp}
		\Expect[|W_x^+|]=\sum_{\actprof\in\actions}\Prob\parens*{\actprof\in W_x^+  }=2^{\nplayers}\cdot \exp\big(-(1+o(1)) \cdot\ldrate(x)\cdot \nplayers \big).
		\end{equation}
		It follows from the definition of $x_{\opt}$ in \cref{eq:def-xopt} that, for all $\varepsilon>0$, there exists some $\delta>0$ such that $\ldrate(x_{\opt}+\varepsilon)>\log(2)+\delta$. Thus, by \cref{eq:1stmom-ldp} and Markov's inequality we conclude
		\begin{equation*}
		\Prob\parens*{W_{x_{\opt}+\varepsilon}^+\neq \varnothing }=\Prob\parens*{|W_{x_{\opt}+\varepsilon}^+|\ge1 }\le \Expect[|W_{x_{\opt}+\varepsilon}^+ |]=e^{-(1+o(1))\delta\cdot\nplayers}\to 0,
		\end{equation*}
		which proves \cref{eq:claim-SO1}. 
		On the other hand,  since, by the definition of $x_{\opt}$ in \cref{eq:def-xopt}, $\ldrate(x_{\opt}-\varepsilon)<\log(2)$ for all $\varepsilon>0$, we have, again by \cref{eq:1stmom-ldp}
		\begin{equation}\label{eq:x-largeavg}
		\Expect[|W_{x_{\opt-\varepsilon}}^+|]\to\infty.
		\end{equation}
		Moreover, for every distinct $\actprof,\actprofalt\in\actions$, thanks to the independence of the payoffs across different profiles, we have
		\begin{equation}
		\forall x\in\R,\quad	\Prob(\actprof,\actprofalt\in W^+_x)=\Prob(\actprof\in W^+_x )^2.
		\end{equation}
		Therefore, for every choice of $x\in\R$
		\begin{align}
		\Expect[|W^+_x|]^2\le \Expect[|W^+_x|^2]&=\sum_{\actprof\in\actions}\sum_{\actprofalt\in\actions}\Prob(\actprof,\actprofalt\in W^+_x)\\
		&=\sum_{\actprof\in\actions}\Prob(\actprof\in W^+_x)+\sum_{\actprof\in\actions}\sum_{\actprofalt\neq\actprof}\Prob(\actprof\in W^+_x)^2\\
		\label{eq:bound2m}	&\le \Expect[|W^+_x|]+\Expect[|W^+_x|]^2.
		\end{align}
		Hence, 
		\begin{equation}\label{eq:2nd-mom-meth}
		\frac{\Expect[|W_{x_{\opt}-\varepsilon}^+|^2]}{\Expect[|W_{x_{\opt}-\varepsilon}^+|]^2}\le \frac{ \Expect[|W^+_{x_{\opt}-\varepsilon}|]+\Expect[|W^+_{x_{\opt}-\varepsilon}|]^2}{\Expect[|W_{x_{\opt}-\varepsilon}^+|]^2} \to 1,
		\end{equation}
		where the first inequality comes from \cref{eq:bound2m} and the limit follows from \cref{eq:x-largeavg}.
		By the second moment method (see \citet[Chapter 4]{AloSpe:Wiley2016}), \cref{eq:2nd-mom-meth} implies \cref{eq:claim-SO2}, since
		\begin{equation*}
		\Prob\parens*{W_{x_{\opt}-\varepsilon}^+ \neq\varnothing }=\Prob\parens*{|W_{x_{\opt}-\varepsilon}^+|\ge 1 }\ge 	\frac{\Expect[|W_{x_{\opt}-\varepsilon}^+|^2]}{\Expect[|W_{x_{\opt}-\varepsilon}^+|]^2}\to 1. 
		\qedhere
		\end{equation*}
	\end{proof}

	\section{First moment computation}\label{sec:first-mom}
	In this section we explicitly compute the expected number of equilibria having an average social utility above/below a given threshold. As in the proof in \cref{sec:SO}, we define the  sets
	\begin{equation}\label{eq:defZ}
	Z^+_x(\game_{\nplayers})\coloneq\left\{\actprof\in\NE(\game_{\nplayers})\::\:\ASU(\game_{\nplayers},\actprof)\ge x \right\}\subseteq W^+_x(\game_{\nplayers})
	\end{equation}
	and
	\begin{equation}
	Z^-_x(\game_{\nplayers})\coloneq\left\{\actprof\in\NE(\game_{\nplayers})\::\:\ASU(\game_{\nplayers},\actprof)\le x \right\}\subseteq W^-_x(\game_{\nplayers}),
	\end{equation}
	namely, $Z_x^+$ ($Z_x^-$) is the set of pure Nash equilibria with \ac{ASU} at least (at most) $x$.

	\begin{lemma}
		\label{le:1st-mom} 
		Let $(Y_{\play})_{\play\le \nplayers}$ be a family of i.i.d. random variables with law $\moddistr$ defined in \cref{eq:def-Ftilde}. Then,
		\begin{equation}
		\forall x\in\R,\qquad\Expect[|Z^+_x|]=(1+\alpha)^{\nplayers}\cdot\Prob\parens*{\frac{1}{\nplayers}\sum_{i=1}^{\nplayers} Y_i\ge x },\quad 	\Expect[|Z^-_x|]=(1+\alpha)^{\nplayers}\cdot\Prob\parens*{\frac{1}{\nplayers}\sum_{i=1}^{\nplayers} Y_i\le x }.
		\end{equation}
	\end{lemma}
	\begin{proof}
		We start by computing the probability that a given strategy profile is a \ac{PNE} with \ac{ASU} larger than a threshold $x\in\R$. The computation for the case in which the \ac{ASU} is smaller than $x$ is the same, since it is sufficient to switch the inequality sign when needed. Notice that, by conditioning,
		\begin{equation}\label{eq:probZ}
		\Prob\parens*{\actprof\in Z^+_x }=\Prob\parens*{\actprof\in\NE }\Prob\parens*{\ASU(\actprof)\ge x\:\big\rvert\: \actprof\in \NE }.
		\end{equation}
		Given two independent random variables $X,X'$ with common distribution $\distr$, 
		\begin{equation}\label{eq:probNE}
		\Prob\parens*{\actprof\in\NE }=\Prob\parens*{\forall \play\in[\nplayers],\: u_{\play}(\actprof)\ge u_{\play}\big(\actprof_{-\play},\act_\play\oplus 1\big) }=\Prob(X\ge X')^{\nplayers}=(1-\beta)^{\nplayers}=\parens*{\frac{1+\alpha}{2}}^{\nplayers}.
		\end{equation}
		Notice that
		\begin{equation}\label{eq:probY}
		\Prob\parens*{\ASU(\actprof)\ge x\:\big\rvert\: \actprof\in \NE }=\Prob\parens*{\frac{1}{\nplayers}\sum_{\play=1}^{\nplayers} Y_\play\ge x },
		\end{equation}
		where $(Y_{\play})_{\play\le \nplayers}$ is a collection of i.i.d. random variables with law $\moddistr$, with
		\begin{align*}
		\moddistr(x)\coloneqq\Prob(X\le x\:|\: X\ge X' ).
		\end{align*}
		By explicit computation
		\begin{align*}
		\moddistr(x)&=\Prob(X\le x\:|\: X\ge X' )=\frac{\Prob(X'\le X\le x)}{\Prob(X\ge X')}\\
		&=\frac{1}{1-\beta}\int_{(-\infty,x]} \Prob(X\in[y,x] ) \diff \distr(y)\\
		&=\frac{1}{1-\beta}\int_{(-\infty,x]} (\distr(x)-\distr(y^-) ) \diff \distr(y)\\
		&=\frac{2}{1+\alpha}\parens*{\distr(x)^2-\int_{(-\infty,x]} \distr(y^-)\diff \distr(y) }.
		\end{align*}
		Furthermore, notice that if $\distr$ is  continuous we have
		\begin{align}\label{eq:def-tildeF-noties}
		\moddistr(x)=\distr(x)^2.
		\end{align}
		With \cref{eq:probZ}, \cref{eq:probNE} and \cref{eq:probY} at hand, we can easily compute the expected size of the set $Z_x^+$,
		\begin{equation}\label{eq:first-mom}
		\Expect[|Z^+_x|]=\Expect\bracks*{\sum_{\actprof\in\actions}\ind_{\actprof\in Z^+_x} }=\sum_{\actprof\in\actions}\Prob(\actprof\in Z^+_x)=(1+\alpha)^{\nplayers}\cdot\Prob\parens*{\frac{1}{\nplayers}\sum_{i=1}^{\nplayers} Y_i\ge x }.\qedhere
		\end{equation}
	\end{proof}

	\section{The case $\alpha=0$}\label{sec:noties}
	We now show that the first moment computation in \cref{sec:first-mom} is enough to show the validity of \cref{eq:result-noties}. 
	\begin{proof}[Proof of \cref{th:convergence-noties}]
		Using the definitions in \cref{eq:defZ}, the claim in  \cref{eq:result-noties} can be rephrased as
		\begin{align}\label{eq:result-noties-1}
		\forall\varepsilon>0,\quad &\Prob\parens*{Z_{x_{\typ}+\varepsilon}^+=\varnothing}\to 1
		\shortintertext{and}
		\label{eq:result-noties-2}
		\forall\varepsilon>0,\quad &\Prob\parens*{Z_{x_{\typ}-\varepsilon}^-=\varnothing}\to 1.
		\end{align} 
		In order to prove the desired result it is enough to use \cref{le:1st-mom} and the law of large numbers applied to an i.i.d. sequence of random variables having law $\moddistr$ as in \cref{eq:def-ftildenoties}. Indeed,
		\begin{equation}
		\Prob\parens*{Z_{x_{\typ}+\varepsilon}^+\neq\varnothing }= \Prob\parens*{|Z_{x_{\typ}+\varepsilon}^+|\ge 1 }\le \Expect\bracks*{|Z_{x_{\typ}+\varepsilon}^+|}=\Prob\left(\frac{1}{\nplayers}\sum_{\play=1}^{\nplayers} Y_\play\ge x_{\typ}+\varepsilon \right)\to 0,
		\end{equation}
		where the inequality follows from Markov's inequality, the last equality follows from \cref{le:1st-mom} and the limit follows from the law of large numbers. Similarly,
		\begin{equation}
		\Prob\parens*{Z_{x_{\typ}-\varepsilon}^-\neq\varnothing }= \Prob\parens*{|Z_{x_{\typ}-\varepsilon}^-|\ge 1 }\le \Expect\bracks*{|Z_{x_{\typ}-\varepsilon}^-|}=\Prob\left(\frac{1}{\nplayers}\sum_{\play=1}^{\nplayers} Y_\play\le x_{\typ}+\varepsilon \right)\to 0.\qedhere
		\end{equation}
	\end{proof}

	%
	%

	\section{The case $\alpha>0$}\label{sec:ties}
	As in \cref{sec:SO,sec:noties}, we start by rewriting the statement of \cref{th:convergence-ties}\ref{it:th:convergence-ties-1} in terms of the sets in \cref{eq:defZ}. Indeed,  \cref{eq:result-ties} is equivalent to the convergences
	\begin{align}
	\label{eq:result-ties-1}
	\forall\varepsilon\neq 0,\quad &\Prob\parens*{Z_{x_{\beq}+\varepsilon}^+=\varnothing}\to \begin{cases} 1&\text{if }\varepsilon>0,\\
	0&\text{if }\varepsilon<0
	\end{cases}
	\shortintertext{and}
	\label{eq:result-ties-2}
	\forall\varepsilon\neq 0,\quad &\Prob\parens*{Z_{x_{\weq}-\varepsilon}^-=\varnothing}\to \begin{cases} 1&\text{if }\varepsilon>0,\\
	0&\text{if }\varepsilon<0.
	\end{cases}
	\end{align} 
	We only present the proof of \cref{eq:result-ties-1}, since the proof of \cref{eq:result-ties-2} follows the same steps. The proof is in line with the one in \cref{sec:SO}, and uses only the first two moments.

	We start by estimating the second moment of the quantity $|Z^{+}_x|$ when $\alpha>0$. 
	Notice that the argument in \cref{eq:bound2m} fails if the set $W_x^+$ is replaced by its subset $Z_x^+$. This is due to the fact that, in general,
	\begin{equation}
	\Prob(\actprof,\actprofalt\in Z_x^+)\neq\Prob(\actprof \in Z_x^+)^2.
	\end{equation}
	Nonetheless, the next lemma shows that the analogue of \cref{eq:2nd-mom-meth} holds even in this case.
	\begin{lemma}\label{le:2nd-mom-est}
		Let $x\in\R$ be such that 
		\begin{equation*}
		\lim_{\nplayers\to\infty}\frac{1}{\nplayers}\log\big(\Expect[|Z_x^+|] \big)>0.
		\end{equation*}
		Then
		\begin{equation}
		\frac{\Expect[|Z_x^+|^2]}{\Expect[|Z_x^+|]^2}\to 1.
		\end{equation}
	\end{lemma}

	\begin{proof}
		We claim that  for every $\actprof,\actprofalt$ differing in at least two strategies, the events $\{\actprof\in Z_x^+\}$ and $\{\actprofalt\in Z_x^+\}$ are independent. 
		Indeed, notice that the event $\{\actprof\in Z_x^+\}$  is measurable with respect to the $\sigma$-field
		\begin{equation}\label{eq:sigma1}
		\sigma\big(\{u_\play(\actprof)\}, \{u_\play(\actprof_{-\play},\act_\play\oplus 1) \}\::\:\play\in[\nplayers]\big).
		\end{equation} 
		We remark that if $\actprofalt$ differs from $\actprof$ in at least two strategies, then the events $\{\actprof\in Z_x^+ \}$ and  $\{\actprofalt\in Z_x^+ \}$ are measurable with respect to independent $\sigma$-fields, hence they are independent. On the other hand, if $\actprof,\actprofalt$ differ in only one strategy, such $\sigma$-fields are not independent. When, for some $\play\in[\nplayers]$,  $\actprofalt=\big(\actprof_{-\play},\act_\play\oplus 1\big)$ we will use the trivial bound
		\begin{align}\label{eq:cond-dist-1}
		\Prob\parens*{\actprof,\actprofalt\in Z_x^+ }\le\Prob\parens*{\actprof\in Z_x^+ }.
		\end{align}
		We now bound the second moment. Let $\mathcal{N}(\actprof)$ denote the set of strategy profiles differing from $\actprof$ in a single coordinate. 
		By \cref{eq:cond-dist-1} we can conclude
		\begin{align*}
		\Expect[|Z_x^+|^2]&=\sum_{\actprof\in\actions}\sum_{\actprofalt\in\actions}\Prob\parens*{\actprof,\actprofalt\in Z_x^+ }\\
		\nonumber&=\sum_{\actprof\in\actions}\bracks*{\Prob(\actprof\in Z_x^+)+\sum_{\actprofalt\in\cN(\actprof)}\Prob(\actprof,\actprofalt\in Z_x^+)+\sum_{\actprofalt\not\in\cN(\actprof)\cup{\{\actprof\}}}\Prob(\actprof\in Z_x^+)^2 }\\
		\nonumber&\le  2^{\nplayers}\cdot\Prob(\actprof\in Z_x^+)+ \nplayers\cdot 2^{\nplayers}\cdot \Prob(\actprof\in Z_x^+) +2^{\nplayers}\cdot\parens*{2^{\nplayers}-\nplayers-1 }\Prob(\actprof\in Z_x^+)^2\\
		\nonumber&=\parens*{1+\nplayers}\Expect[|Z^+_x|]+(1+o(1))\big(2^{\nplayers}\cdot\Prob(\actprof\in Z_x^+)\big)^2\\
		=&\big(1+o(1)\big)\cdot\Expect[|Z_x^+|]^2+O\parens*{{\nplayers}\cdot \Expect[|Z_x^+|]}.
		\end{align*}
		We notice that, if the expected size of $Z_x^+$ is exponentially large, then its square is asymptotically larger than ${\nplayers}\cdot \Expect[|Z_x^+|]$, and therefore the second moment of $|Z_x^+|$ coincides at first order with the first moment square. 
		More precisely, if for some $x\in\R$ we have
		\begin{equation}\label{eq:sec-mom-ties}
		\lim_{\nplayers\to\infty}\frac{1}{\nplayers}\log\big(\Expect[|Z_x^+|]\big)>0\quad\text{then}\quad\frac{\Expect[|Z_x^+|^2]}{\Expect[|Z_x^+|]^2}\to 1,
		\end{equation}
		which is the desired result.
	\end{proof}

	Using the results in \cref{le:1st-mom} and \cref{le:2nd-mom-est} we now prove the convergence in \cref{eq:result-ties-1}. 
	As in \cref{sec:SO}, the second case in \cref{eq:result-ties-1} follows by Markov's inequality. To show the first case in \cref{eq:result-ties-1}, we need to consider the large deviation rate for the distribution $\moddistr$ in \cref{eq:def-Ftilde}, i.e.,
	\begin{equation}\label{eq:def-ratetilde}
	\modldrate(x)\coloneqq\sup_{t\in\R}\big[xt-\log\parens*{\Expect\bracks*{e^{tY}} } \big],
	\end{equation}
	where $Y\sim\moddistr$. The fact that the rate in \cref{eq:def-ratetilde} is well defined follows from the next lemma.
	\begin{lemma}\label{le:mgf}
		Let $\distr$ be a distribution on $\R$ such that, if $X\sim \distr$,
		\begin{equation*}
		\Expect[e^{tX}]\le \infty,\quad\forall t\in\R,
		\end{equation*}
		and let $\moddistr$ be defined as in \cref{eq:def-Ftilde} and $Y\sim\moddistr$, then
		\begin{equation*}
		\Expect[e^{tY}]\le 2\Expect[e^{tX}],\quad\forall t\in\R.
		\end{equation*}
	\end{lemma}
	
	\begin{proof}
		Let $X,X'$ be two independent random variables with law $\distr$. Recall that
		\begin{equation}
		\moddistr(x)=\Prob\parens*{X\le x\:|\: X\ge X' }.
		\end{equation}
		Therefore
		\begin{align*}
		\Expect[e^{tY}]&=\Expect[e^{tX}\:|\: X\ge X']\\
		&=\parens*{1-\sum_{\ell\in L}\sqrt{\alpha_\ell}}\int_{\R}e^{tx}f(x|X\ge X')\diff x+\sum_{\ell\in L}\sqrt{\alpha_\ell}e^{t\ell}\Prob\parens*{X=\ell\:|\: X\ge X' }\\
		&=\frac{1}{1-\beta}\bracks*{\parens*{1-\sum_{\ell\in L}\sqrt{\alpha_\ell}}\int_{\R }e^{tx}f(x)\distr(x)\diff x+\sum_{\ell\in L}\sqrt{\alpha_\ell} \distr(\ell)e^{t\ell} }\\
		&\le \frac{1}{1-\beta}\bracks*{\parens*{1-\sum_{\ell\in L}\sqrt{\alpha_\ell}}\int_{\R }e^{tx}f(x)\diff x+\sum_{\ell\in L}\sqrt{\alpha_\ell}e^{t\ell}  }\\
		&\le 2\Expect[e^{tX}],
		\end{align*}
		where in the first inequality we used the trivial bound $\distr(x)\le 1$ for all $x\in\R$, while in the last inequality we used the fact that $\beta\le 1/2$. 
	\end{proof}	
	\begin{proof}[Proof of \cref{th:convergence-ties}\ref{it:th:convergence-ties-1}]
		By \cref{le:1st-mom} and the definition of $x_{\beq}$ in \cref{eq:def-xbeq},
		\begin{equation*}
		\forall\varepsilon>0,\qquad\Expect[|Z_{x_{\beq}+\varepsilon}^+|]=(1+\alpha)^{\nplayers} \exp\big(-(1+o(1))\cdot \modldrate(x)\cdot\nplayers \big)\to 0,
		\end{equation*}
		and by Markov's inequality,
		\begin{equation*}
		\Prob\parens*{Z_{x_{\beq}+\varepsilon}^+\neq \varnothing }\le \Expect[|Z_{x_{\beq}+\varepsilon}^+ |]\to 0.
		\end{equation*}
		On the other hand, again by \cref{le:1st-mom},
		\begin{equation}
		\lim_{\nplayers\to\infty}\frac{1}{\nplayers}\log\big(\Expect[|Z_{x_{\beq}-\varepsilon}^+|]\big)>0,
		\end{equation}
		so that, by \cref{le:2nd-mom-est} and the second moment method, we can conclude that
		\begin{equation}
		\Prob\parens*{Z_x^+ \neq\varnothing }\ge 	\frac{\Expect[|Z_x^+|]^2}{\Expect[|Z_x^+|^2]}\to 1.\qedhere
		\end{equation}
	\end{proof}
	We are no left to show the validity of \cref{eq:result-typ}.


	\begin{proof}[Proof of \cref{th:convergence-ties}\ref{it:th:convergence-ties-2}]
		Using the notation adopted in the proofs, the convergence in \cref{eq:result-typ} can be rephrased as
		\begin{equation}\label{eq:result-typ-new}
		\forall\varepsilon>0,\quad\frac{|Z^+_{x_{\typ}+\varepsilon}|+|Z^-_{x_{\typ}-\varepsilon}|}{|\NE|}\overset{\Prob}{\longrightarrow}0.
		\end{equation} 
		We now show that
		\begin{equation}\label{eq:result-typ-new+}
		\forall\varepsilon>0,\quad\frac{|Z^+_{x_{\typ}+\varepsilon}|}{|\NE|}\overset{\Prob}{\longrightarrow}0,
		\end{equation} 
		and the result for $Z_{x_{\typ}-\varepsilon}^-$ follows from the same argument by reversing the signs. 
		
		Notice that, by \cref{eq:amiet1},
		\begin{equation}\label{eq:amiet}
		\forall\delta>0,\quad \Prob\big(|\NE|>(1+\alpha-\delta )^{\nplayers}  \big)\to 1.
		\end{equation} 
		Therefore, \cref{eq:result-typ-new+} is equivalent to that of the following statement
		\begin{equation}\label{eq:result-typ-new++}
		\forall\varepsilon>0,\quad\exists\delta=\delta(\varepsilon)>0\:\text{ s.t.}\quad \Prob\parens*{|Z_{x_{\typ}+\varepsilon}^+|<(1+\alpha-\delta)^{\nplayers} }\to 1.
		\end{equation}
		Fix some $\varepsilon>0$ and consider the case when the expected size of $Z_{x_{\typ}+\varepsilon}^+$ is not exponentially large, i.e.,
		\begin{equation}
		\lim_{\nplayers\to\infty}\frac{1}{\nplayers}\log\big(\Expect[|Z_{x_{\typ}+\varepsilon}^+|] \big)\le 0.
		\end{equation}
		In this case, by Markov's inequality, for all $\delta\in(0,\alpha)$
		\begin{equation}
		\Prob\parens*{|Z_{x_{\typ}+\varepsilon}^+|>(1+\alpha-\delta)^{\nplayers} }\le \frac{\Expect[|Z_{x_{\typ}+\varepsilon}^+|]}{(1+\alpha-\delta)^{\nplayers}}\to 0.
		\end{equation}
		Consider now $\varepsilon>0$ such that
		\begin{equation}\label{eq:gamma}
		\lim_{\nplayers\to\infty}\frac{1}{\nplayers}\log\big(\Expect[|Z_{x_{\typ}+\varepsilon}^+|] \big)=\log(1+\alpha)-\modldrate(x_{\typ}+\varepsilon)\in \big(0,\log(1+\alpha)\big),
		\end{equation}
		where the first equality follows by \cref{le:1st-mom} and Cramer's large deviation theorem. Even if the expected size of $Z_{x_{\typ}+\varepsilon}^+$ is exponentially large, by \cref{eq:amiet} it is still smaller than the size of $\NE$. Therefore, to prove that \cref{eq:result-typ-new++} holds it is enough to show that $|Z_{x_{\typ}+\varepsilon}^+|$  concentrates at first order around its expectation. By applying Chebyshev's inequality we get that, for all $\gamma>0$,
		\begin{align*}
		\Prob\parens*{\big| |Z_{x_{\typ}+\varepsilon}^+|- \Expect[|Z_{x_{\typ}+\varepsilon}^+|] \big|> \gamma \Expect\bracks{|Z_{x_{\typ}+\varepsilon}^+|} }\le \frac{\Var\bracks*{|Z_{x_{\typ}+\varepsilon}^+|}}{\gamma^2 \Expect\bracks*{|Z_{x_{\typ}+\varepsilon}^+|}^2}
		\end{align*}
		By \cref{eq:gamma} and \cref{le:2nd-mom-est} we know that
		\begin{equation}
		\Var\bracks*{|Z_{x_{\typ}+\varepsilon}^+|}=\Expect[|Z_{x_{\typ}+\varepsilon}^+|^2]-\Expect[|Z_{x_{\typ}+\varepsilon}^+|]^2=o\big(\Expect[|Z_{x_{\typ}+\varepsilon}^+|]^2\big),
		\end{equation}
		hence,
		\begin{align*}
		\forall\gamma>0,\quad\Prob\parens*{\big| |Z_{x_{\typ}+\varepsilon}^+|- \Expect[|Z_{x_{\typ}+\varepsilon}^+|] \big|> \gamma \Expect[|Z_{x_{\typ}+\varepsilon}^+|] }\to 0,
		\end{align*}
		and, as a consequence,
		\begin{align*}
		\Prob\parens*{|Z_{x_{\typ}+\varepsilon}^+|> 2 \Expect[|Z_{x_{\typ}+\varepsilon}^+|]  }\to 0.
		\end{align*}
		Since, by \cref{eq:gamma}, there exists some $\delta=\delta(\varepsilon)>0$ such that
		\begin{equation*}
		\Expect[|Z_{x_{\typ}+\varepsilon}^+|]< (1-\alpha-\delta)^{\nplayers},
		\end{equation*}
		\cref{eq:result-typ-new++} immediately follows.
	\end{proof}

	\section{binary payoffs}\label{sec:binary}
	\label{se:proof-convergence}
	In this section we study in detail the case in which $\distr$ is the distribution of a Bernoulli random variable with parameter $\proba$, i.e.,
	\begin{equation*}
	\distr(x)=\begin{cases}0&\text{if }x<0\\
	1-\proba&\text{if }x\in[0,1) \\
	1&\text{if }x\ge 1
	\end{cases}.
	\end{equation*}
	It follows that
	\begin{equation*}
	\alpha=\proba^2+(1-\proba)^2
	\end{equation*}
	and
	\begin{equation*}
	\Prob(X=0\:|\: X\ge X')=\frac{1}{1-\beta}\Prob(X=X'=0)=\frac{(1-\proba)^2}{1-\proba(1-\proba)}=\frac{1-2\proba+\proba^2}{1-\proba+\proba^2},
	\end{equation*}
	so that
	\begin{equation*}
	\moddistr(x)=\begin{cases}0&\text{if }x<0\\
	\frac{1-2\proba+\proba^2}{1-\proba+\proba^2}&\text{if }x\in[0,1) \\
	1&\text{if }x\ge 1
	\end{cases},
	\end{equation*}
	i.e., $\moddistr$ is the distribution of a Bernoulli random variable with parameter
	\begin{equation}\label{ptilde}
	\modproba= 1-\frac{1-2\proba+\proba^2}{1-\proba+\proba^2}=\frac{\proba}{1-\proba+\proba^2}.
	\end{equation}
	For $q\in(0,1)$ and $x\in[0,1]$ consider the entropy function
	\begin{equation*}
	H_q(x)\coloneq x\log\parens*{\frac{x}{q} }+(1-x)\log\parens*{\frac{1-x}{1-q} }.
	\end{equation*}
	Then, the large deviation rates $\ldrate$ and $\modldrate$ in \cref{eq:def-rate,eq:def-ratetilde} have the following form
	\begin{equation*}
	\forall \proba\in(0,1),x\in[0,1],\quad \ldrate(x)=H_\proba(x)\quad\text{and}\quad\modldrate(x)=H_{\modproba}(x),
	\end{equation*}
	\citep[see, e.g.,][page~3]{Fis:mimeo2013}.
	Notice that $H_q(x)$ is convex in $x$ for all $q\in(0,1)$, and assumes its minimum at $x=q$, where $H_q(q)=0$.
	\begin{figure}[h]
		\centering
		\includegraphics[width=6.2cm]{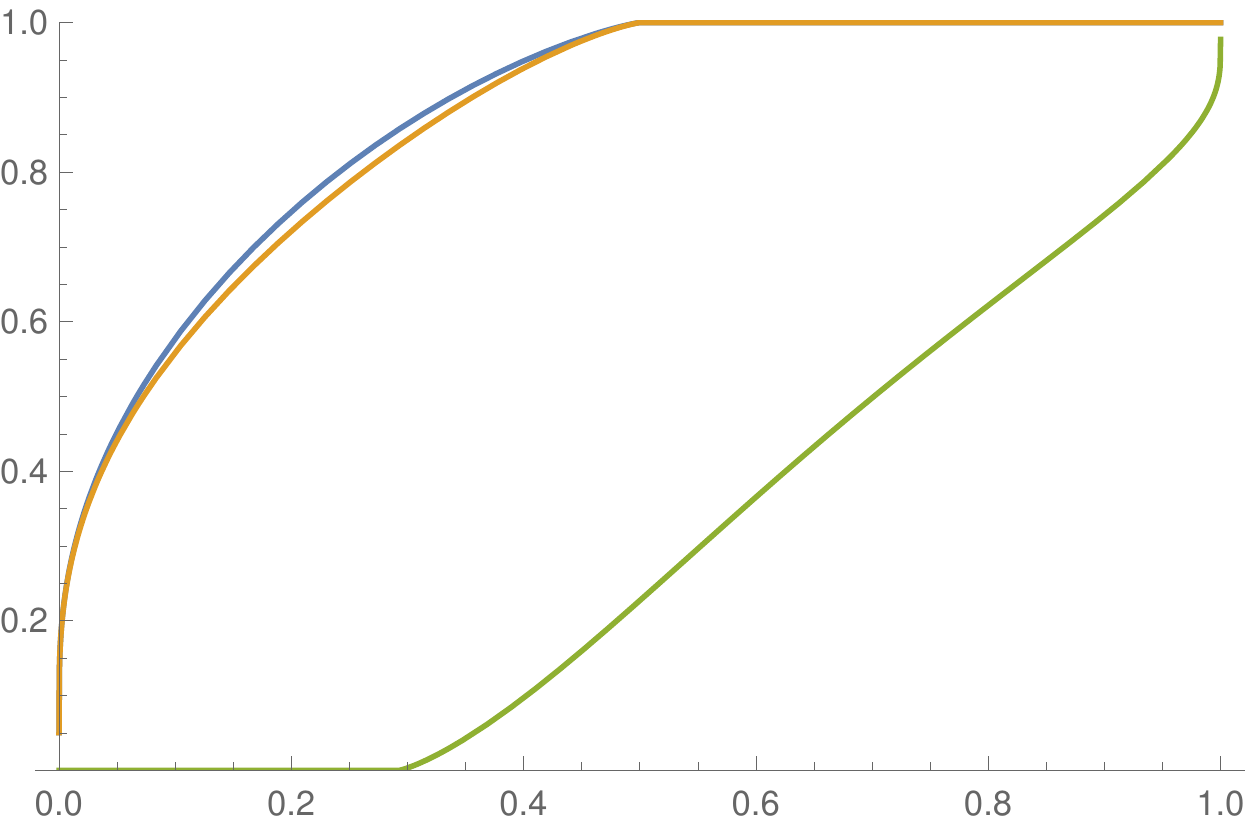}\qquad\qquad
		\includegraphics[width=6.2cm]{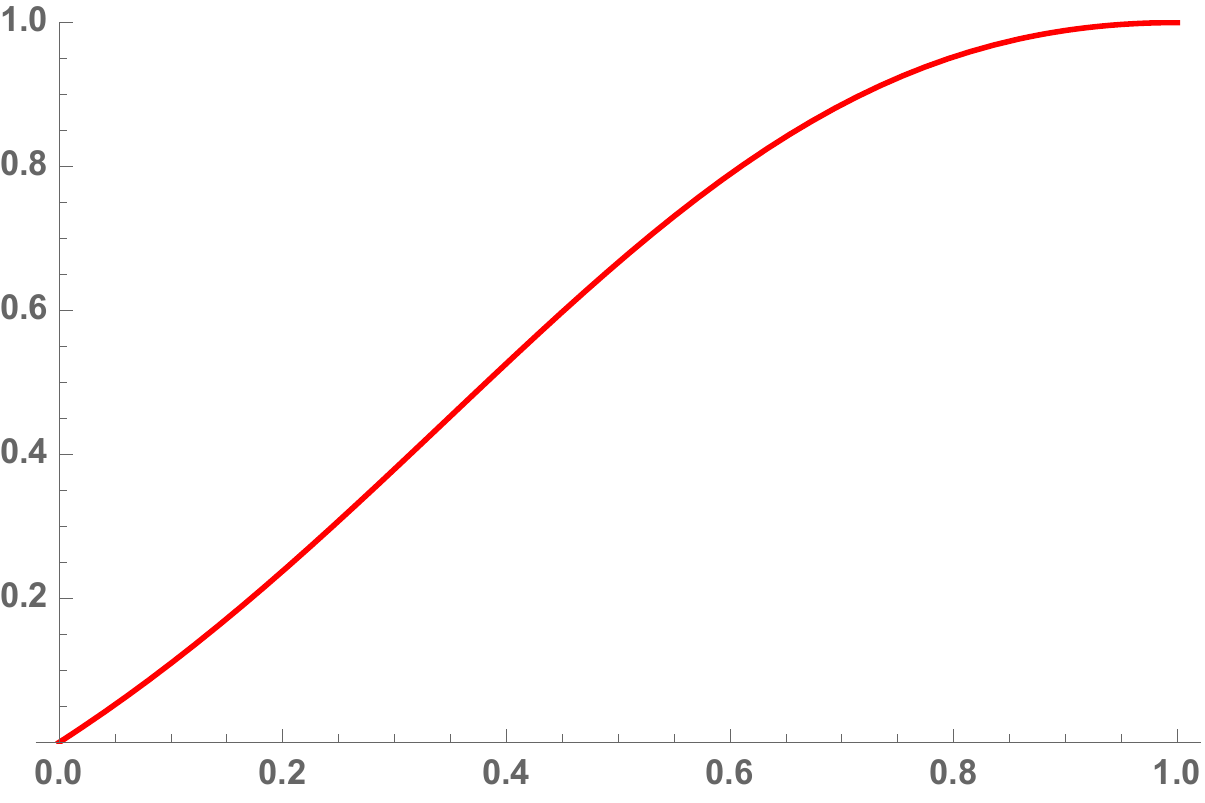}
		\caption{Left: numerical approximation of the functions defined in \cref{eq:x-x-x}. In blue: $x_{\opt}(\proba)$. In orange: $x_{\beq}(\proba)$. In green: $x_{\weq}(\proba)$. Right: Plot of the function $x_{\typ}(\proba)=\proba/(1-\proba+\proba^2)$.}\label{fig:x-x-x}
	\end{figure}
	\begin{proposition}\label{pr:binary}
		Let $\distr$ be the distribution of a Bernoulli random variable with parameter $\proba$. Then there exists three continuous functions
		\begin{equation}
		\label{eq:x-x-x}
		x_{\opt},x_{\beq},x_{\weq}:(0,1)\to[0,1],
		\end{equation}
		such that
		\begin{equation}\label{eq:bin-conv}
		\big(\SO(\game_{\nplayers}),\BEq(\game_{\nplayers}),\WEq(\game_{\nplayers}) \big)\overset{\Prob}{\longrightarrow}\big(x_{\opt}(\proba),x_{\beq}(\proba),x_{\weq}(\proba) \big).
		\end{equation}
		Moreover, the functions $x_{\opt}$ and $x_{\beq}$ are both increasing on the interval $(0,1/2)$ and are identically equal to $1$ on the interval $[1/2,1]$.
		The function $x_{\weq}$ is identically $0$ on the interval $(0,1-\sqrt{2}/2)$ and is increasing on the interval $[1-\sqrt{2}/2,1]$.
	\end{proposition}
	In words, the limit quantities for $\SO$, $\BEq$ and $\WEq$---seen as a functions of $\proba$---display a peculiar behavior, namely, there exists some threshold for the value of $\proba$ before/after which the functions stay constant.
	\begin{proof}[Proof of \cref{pr:binary}]
		The convergence in \cref{eq:bin-conv} is a corollary of \cref{pr:SO} and \cref{th:convergence-ties}. More explicitly, the three limiting quantities can be defined as
		\begin{align}
		\label{eq:xopt-p}	x_{\opt}(\proba)&\coloneqq\inf\left\{x>\proba\::\: H_\proba(x)= \log(2) \right\},\\
		\label{eq:xbeq-p}		x_{\beq}(\proba)&\coloneqq\sup\left\{x\in[0,1]\::\: H_{\modproba}(x)= \log(1+\proba^2+(1-\proba)^2) \right\},\\
		\label{eq:xweq-p}			x_{\weq}(\proba)&\coloneqq\inf\left\{x\in[0,1]\::\: H_{\modproba}(x)= \log(1+\proba^2+(1-\proba)^2) \right\},
		\end{align} 
		where $\modproba$ is defined as in \cref{ptilde}.
		In order to prove the continuity and monotonicity of the functions as stated in \cref{pr:binary} it is sufficient to recall that both $H_\proba(x)$ and $H_{\modproba}(x)$ are continuous and convex in $x$ and they assume their minimum value, i.e., $0$, only at $\proba$ and $\modproba$, respectively. See \cref{fig2} for a plot of the two function for different choices of $\proba$. Notice also that the function $\proba\mapsto H_{\modproba}(0)$ is increasing in $\proba$ and that, by definition of $H_{\modproba}(x)$ and $\alpha$,
		\begin{equation}
		H_{\modproba}(0)\le \log(1+\alpha)\quad\iff \quad (1-\proba)(1+\proba^2+(1-\proba)^2)\ge 1\quad\iff\quad \proba\le 1-\frac{\sqrt{2}}{2}.
		\end{equation}
		The latter implies that $x_{\weq}(\proba)=0$ for all $\proba\le 1-\frac{\sqrt{2}}{2}$, while $x_{\weq}(\proba)$ is strictly increasing for $\proba\ge 1-\frac{\sqrt{2}}{2}$.
		
		Similarly, the functions $\proba\mapsto H_{\modproba}(1)$ and $\proba\mapsto H_{\proba}(1)$ are decreasing in $\proba$ and
		\begin{equation}
		H_{\proba}(1)\le \log(2)\quad\iff\quad \proba\ge \frac{1}{2},
		\end{equation}
		while
		\begin{equation}
		H_{\modproba}(1)\le \log(1+\alpha)\quad\iff\quad\frac{1-\proba+\proba^2}{\proba}\le 1+\proba^2+(1-\proba)^2   \quad\iff\quad \proba\ge \frac{1}{2}.
		\end{equation}
		In other words, $x_{\opt}(\proba)=x_{\beq}(\proba)=1$ for all $\proba\ge 1/2$, while  they are increasing functions of $\proba$ in the interval $[0,1/2]$.
	\end{proof} 	
	\begin{figure}[h]
		\centering
		\includegraphics[width=6.2cm]{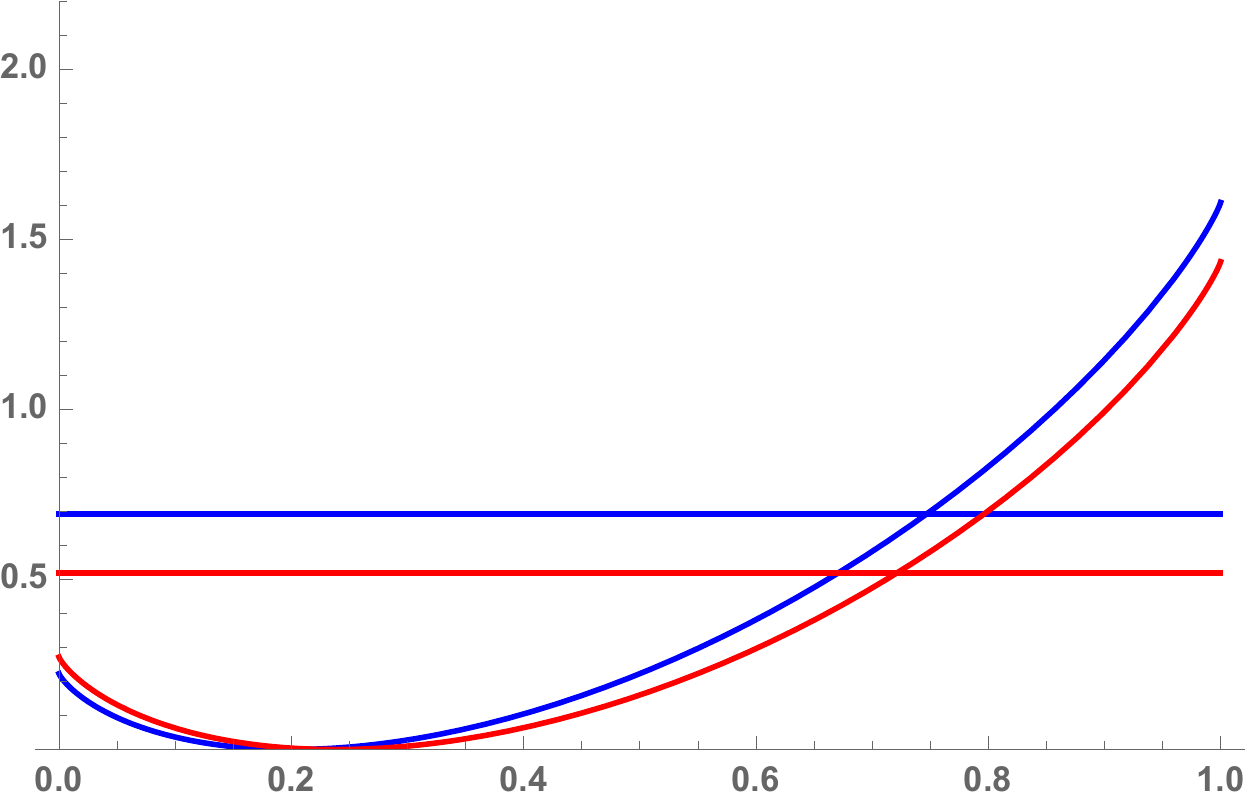}\qquad\qquad
		\includegraphics[width=6.2cm]{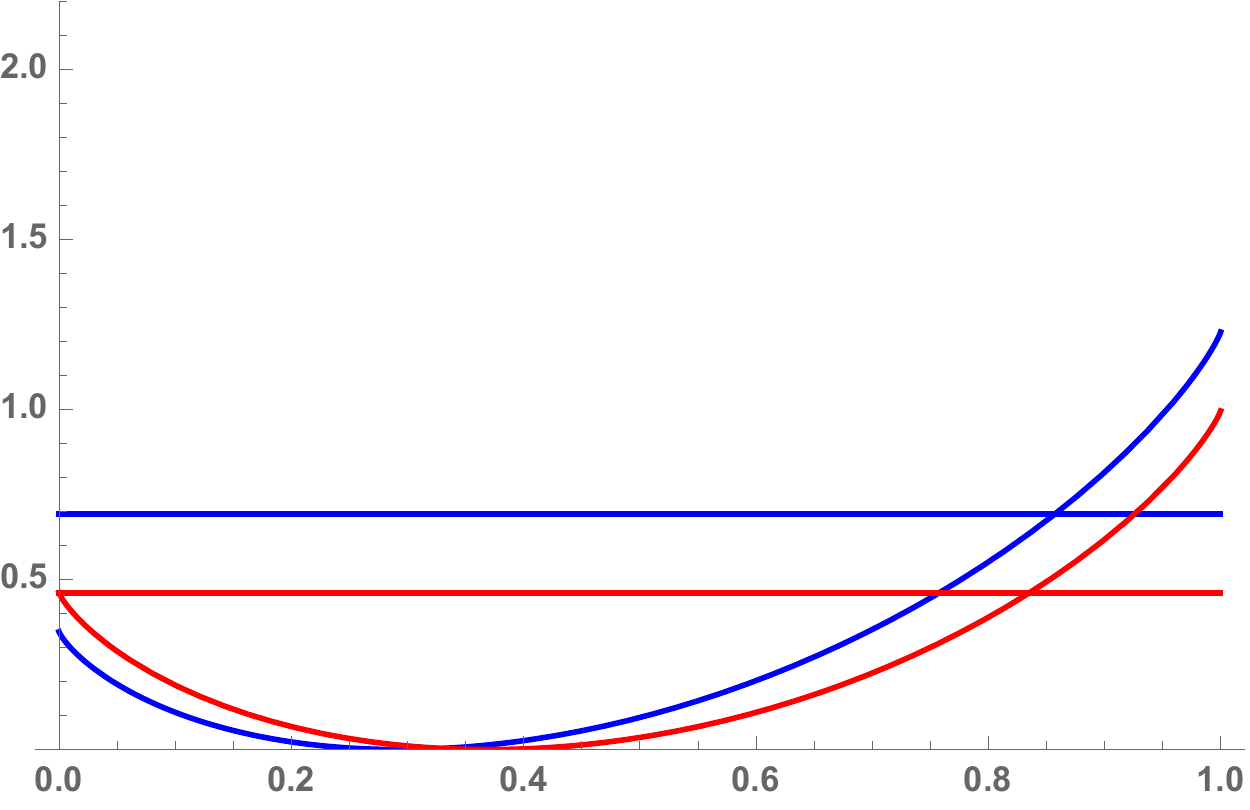}\\ \vspace{0.5cm}
		\includegraphics[width=6.2cm]{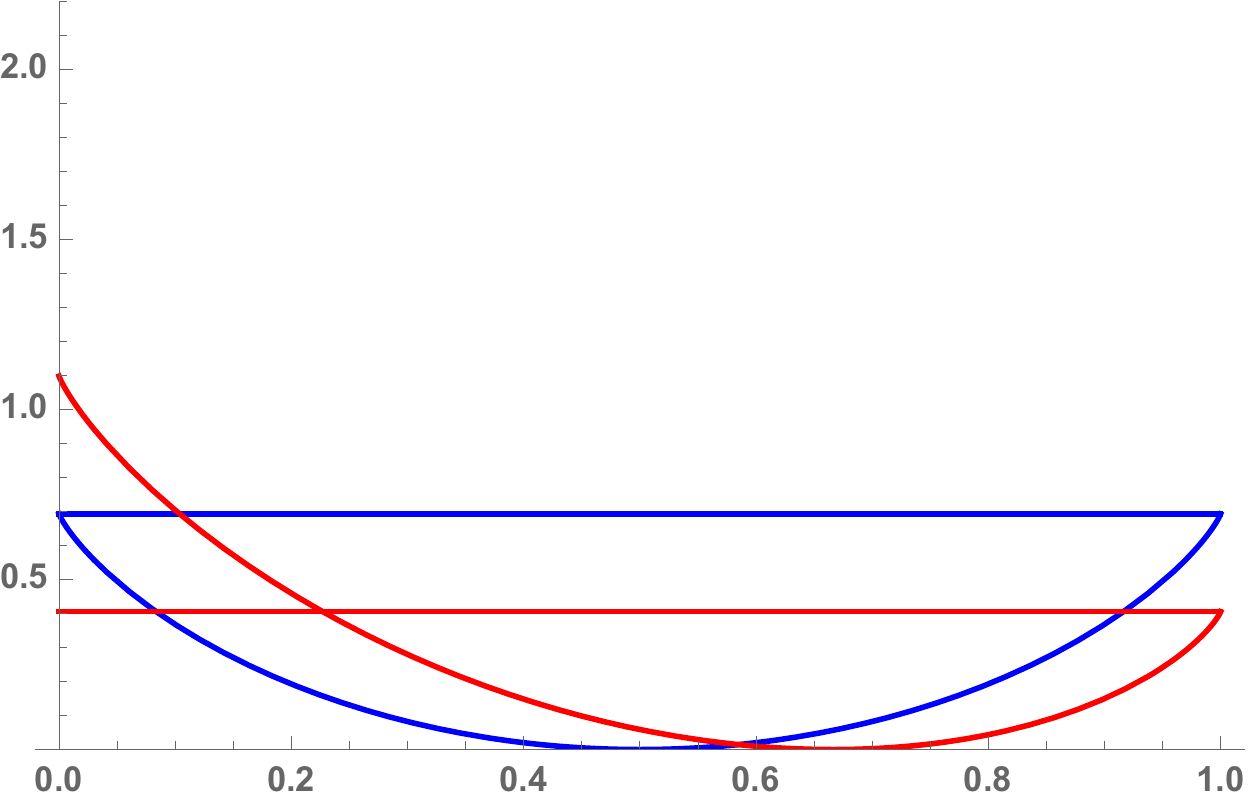}\qquad\qquad
		\includegraphics[width=6.2cm]{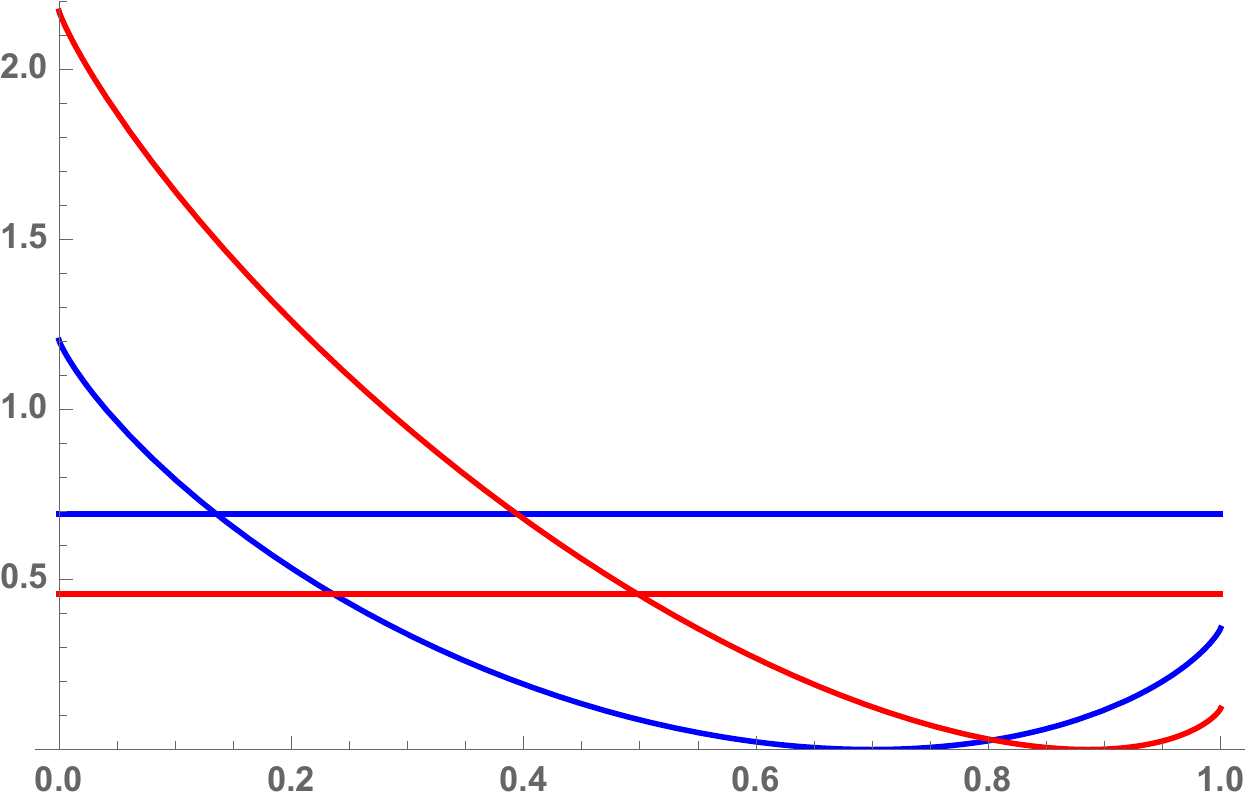}
		\caption{In this figure we present the plot of the functions $H_\proba(x)$ (in blue) and $H_{\modproba}(x)$ (in red) for different values of $\proba$. The height of the horizontal lines is $\log(2)$ for the blue line and $\log(1+\alpha)$ for the red one. The value of $\proba$ in the different plots is (from top left to bottom right) $\proba=\frac{1}{5},1-\frac{\sqrt{2}}{2},\frac{1}{2},\frac{7}{10}$.}\label{fig2}
	\end{figure}
	Notice that by choosing $\proba=1/2$, the game $\game_{\nplayers}$ is a drawn uniformly from the space of games with $\nplayers$ players and binary payoffs. In other words, claiming that some property holds with probability approaching $1$ in the model in \cref{pr:binary} with $\proba=1/2$, is equivalent to claim that the fraction of games with binary payoffs sharing that property approaches $1$ as $\nplayers$ grows to infinity. Therefore, choosing $\proba=1/2$, we can rephrase \cref{pr:binary} as a \emph{counting problem} and obtain the following result.
	\begin{corollary}\label{corollary}
		Let $\cG_{\nplayers}$ be the set of all possible distinct games with  $\nplayers$ players and binary payoff. For $\varepsilon>0$ let
		\begin{equation*}
		\widetilde{\cG}_{\nplayers,\varepsilon}\coloneq\left\{\game_{\nplayers}\in\cG_{\nplayers}\::\: \WEq(\game_{\nplayers})\in[x_{\weq}(1/2)-\varepsilon,x_{\weq}(1/2) +\varepsilon ],\:\SO(\game_{\nplayers}),\BEq(\game_{\nplayers})\in[1-\varepsilon,1] \right\},
		\end{equation*}
		that is, the subset of games $\game_{\nplayers}$ with $\SO$, $\BEq$ and $\WEq$ at most $\varepsilon$ far from $x_{\opt}(1/2)$, $x_{\beq}(1/2)$ and $x_{\weq}(1/2)$.	Then,
		\begin{equation*}
		\forall\varepsilon>0,\quad\lim_{\nplayers\to\infty}\frac{|\widetilde{\cG}_{\nplayers,\varepsilon}|}{|\cG_{\nplayers}|}=1.
		\end{equation*}
	\end{corollary}
	\noindent Roughly, \cref{corollary} states that asymptotically almost every binary game $\game_{\nplayers}$ has
	\begin{equation}\label{eq:approx-p}
	(\SO(\game_{\nplayers}),\BEq(\game_{\nplayers}),\WEq(\game_{\nplayers}))\approx (1,1,0.2271),
	\end{equation}
	where the approximation $x_{\weq}(1/2)\approx 0.2271$ can be obtained numerically from the definition of $x_{\weq}$ in \cref{eq:xweq-p}. In the language of \ac{PoA}/\ac{PoS}, the claim of \cref{corollary} can be rephrased as follows: when the number of players grows to infinity for all but a vanishingly small fraction of games with binary payoffs it holds 
	\begin{equation*}
	\PoS(\game_{\nplayers})=\frac{\SO(\game_{\nplayers})}{\BEq(\game_{\nplayers})}\approx 1\quad\text{and}\quad \PoA(\game_{\nplayers})=\frac{\SO(\game_{\nplayers})}{\WEq(\game_{\nplayers})}\approx\frac{1}{0.2271}\approx4.4034.
	\end{equation*}
	
%
	
\bigskip
\subsection*{Acknowledgments}
Both authors are members of GNAMPA-INdAM and of COST Action GAMENET. This work was partially supported by the GNAMPA-INdAM Project 2020 ``Random walks on random games'' and PRIN 2017 project ALGADIMAR.
\bigskip

\bibliographystyle{apalike}
\bibliography{../bibtex/bibrandomPoA}

\begin{thebibliography}{}

\bibitem[Alon and Spencer, 2016]{AloSpe:Wiley2016}
Alon, N. and Spencer, J.~H. (2016).
\newblock {\em The Probabilistic Method}.
\newblock Wiley Series in Discrete Mathematics and Optimization. John Wiley \&
  Sons, Inc., Hoboken, NJ, fourth edition.

\bibitem[Amiet et~al., 2021a]{AmiColHam:ORL2021}
Amiet, B., Collevecchio, A., and Hamza, K. (2021a).
\newblock When ``{B}etter'' is better than ``{B}est''.
\newblock {\em Oper. Res. Lett.}, 49(2):260--264.

\bibitem[Amiet et~al., 2021b]{AmiColSca:arXiv2019}
Amiet, B., Collevecchio, A., and Scarsini, M. (2021b).
\newblock Pure {N}ash equilibria and best-response dynamics in random games.
\newblock {\em Math. Oper. Res.}, forthcoming.

\bibitem[Anshelevich et~al., 2008]{AnsDasKleTarWexRou:SIAMJC2008}
Anshelevich, E., Dasgupta, A., Kleinberg, J., Tardos, E., Wexler, T., and
  Roughgarden, T. (2008).
\newblock The price of stability for network design with fair cost allocation.
\newblock {\em SIAM J. Comput.}, 38(4):1602--1623.

\bibitem[Cohen, 1998]{Coh:PNAS1998}
Cohen, J.~E. (1998).
\newblock Cooperation and self-interest: {P}areto-inefficiency of {N}ash
  equilibria in finite random games.
\newblock {\em Proc. Natl. Acad. Sci. USA}, 95(17):9724--9731.

\bibitem[Daskalakis et~al., 2011]{DasDimMos:AAP2011}
Daskalakis, C., Dimakis, A.~G., and Mossel, E. (2011).
\newblock Connectivity and equilibrium in random games.
\newblock {\em Ann. Appl. Probab.}, 21(3):987--1016.

\bibitem[den Hollander, 2000]{den:AMS2000}
den Hollander, F. (2000).
\newblock {\em Large deviations}, volume~14 of {\em Fields Institute
  Monographs}.
\newblock American Mathematical Society, Providence, RI.

\bibitem[Dresher, 1970]{Dre:JCT1970}
Dresher, M. (1970).
\newblock Probability of a pure equilibrium point in {$n$}-person games.
\newblock {\em J. Combinatorial Theory}, 8:134--145.

\bibitem[Durand and Gaujal, 2016]{DurGau:AGT2016}
Durand, S. and Gaujal, B. (2016).
\newblock Complexity and optimality of the best response algorithm in random
  potential games.
\newblock In {\em Algorithmic Game Theory}, volume 9928 of {\em Lecture Notes
  in Comput. Sci.}, pages 40--51. Springer, Berlin.

\bibitem[Fischer, 2013]{Fis:mimeo2013}
Fischer, M. (2013).
\newblock Large deviations, weak convergence, and relative entropy.
\newblock Technical report, Universit{\`a} di Padova.

\bibitem[Galla and Farmer, 2013]{GalFar:PNAS2013}
Galla, T. and Farmer, J.~D. (2013).
\newblock Complex dynamics in learning complicated games.
\newblock {\em Proc. Natl. Acad. Sci. USA}, 110(4):1232--1236.

\bibitem[Goldberg et~al., 1968]{GolGolNew:JRNBSB1968}
Goldberg, K., Goldman, A.~J., and Newman, M. (1968).
\newblock The probability of an equilibrium point.
\newblock {\em J. Res. Nat. Bur. Standards Sect. B}, 72B:93--101.

\bibitem[Goldman, 1957]{Gol:AMM1957}
Goldman, A.~J. (1957).
\newblock The probability of a saddlepoint.
\newblock {\em Amer. Math. Monthly}, 64:729--730.

\bibitem[Heinrich et~al., 2021]{HeiJanMunPanScoTarWie:arxiv2021}
Heinrich, T., Jang, Y., Mungo, L., Pangallo, M., Scott, A., Tarbush, B., and
  Wiese, S. (2021).
\newblock Best-response dynamics, playing sequences, and convergence to
  equilibrium in random games.
\newblock Technical report, arXiv:2101.04222.

\bibitem[Koutsoupias and Papadimitriou, 1999]{KouPap:STACS1999}
Koutsoupias, E. and Papadimitriou, C. (1999).
\newblock Worst-case equilibria.
\newblock In {\em S{TACS} 99 ({T}rier)}, volume 1563 of {\em Lecture Notes in
  Comput. Sci.}, pages 404--413. Springer, Berlin.

\bibitem[Nash, 1951]{Nash:AM1951}
Nash, J. (1951).
\newblock Non-cooperative games.
\newblock {\em Ann. of Math. (2)}, 54:286--295.

\bibitem[Nash, 1950]{Nash:PNAS1950}
Nash, Jr., J.~F. (1950).
\newblock Equilibrium points in {$n$}-person games.
\newblock {\em Proc. Nat. Acad. Sci. U. S. A.}, 36:48--49.

\bibitem[Pangallo et~al., 2019]{PanHeiFar:SA2019}
Pangallo, M., Heinrich, T., and Doyne~Farmer, J. (2019).
\newblock Best reply structure and equilibrium convergence in generic games.
\newblock {\em Science Advances}, 5(2).

\bibitem[Papadimitriou, 2001]{Pap:ACMSTC2001}
Papadimitriou, C. (2001).
\newblock Algorithms, games, and the {Internet}.
\newblock In {\em Proceedings of the {T}hirty-{T}hird {A}nnual {ACM}
  {S}ymposium on {T}heory of {C}omputing}, pages 749--753, New York. ACM.

\bibitem[Pigou, 1920]{Pig:Macmillan1920}
Pigou, A.~C. (1920).
\newblock {\em The Economics of Welfare}.
\newblock Macmillan and Co., London.

\bibitem[Powers, 1990]{Pow:IJGT1990}
Powers, I.~Y. (1990).
\newblock Limiting distributions of the number of pure strategy {N}ash
  equilibria in {$N$}-person games.
\newblock {\em Internat. J. Game Theory}, 19(3):277--286.

\bibitem[Rai\v{c}, 2003]{Rai:P7YSM2003}
Rai\v{c}, M. (2003).
\newblock Normal approximation by {S}tein's method.
\newblock In Mrvar, A., editor, {\em Proceedings of the Seventh Young
  Statisticians Meeting}, pages 71--97.

\bibitem[Rinott and Scarsini, 2000]{RinSca:GEB2000}
Rinott, Y. and Scarsini, M. (2000).
\newblock On the number of pure strategy {N}ash equilibria in random games.
\newblock {\em Games Econom. Behav.}, 33(2):274--293.

\bibitem[Roughgarden and Tardos, 2007]{RouTar:inCUPress2007}
Roughgarden, T. and Tardos, {\'E}. (2007).
\newblock Introduction to the inefficiency of equilibria.
\newblock In {\em Algorithmic Game Theory}, pages 443--459. Cambridge Univ.
  Press, Cambridge.

\bibitem[Schulz and Stier~Moses, 2003]{SchSti:P14SIAM2003}
Schulz, A.~S. and Stier~Moses, N. (2003).
\newblock On the performance of user equilibria in traffic networks.
\newblock In {\em Proceedings of the {F}ourteenth {A}nnual {ACM}-{SIAM}
  {S}ymposium on {D}iscrete {A}lgorithms ({B}altimore, {MD}, 2003)}, pages
  86--87, New York. ACM.

\bibitem[Stanford, 1995]{Sta:GEB1995}
Stanford, W. (1995).
\newblock A note on the probability of {$k$} pure {N}ash equilibria in matrix
  games.
\newblock {\em Games Econom. Behav.}, 9(2):238--246.

\bibitem[Stanford, 1996]{Sta:MOR1996}
Stanford, W. (1996).
\newblock The limit distribution of pure strategy {N}ash equilibria in
  symmetric bimatrix games.
\newblock {\em Math. Oper. Res.}, 21(3):726--733.

\bibitem[Stanford, 1997]{Sta:MSS1997}
Stanford, W. (1997).
\newblock On the distribution of pure strategy equilibria in finite games with
  vector payoffs.
\newblock {\em Math. Social Sci.}, 33(2):115--127.

\bibitem[Stanford, 1999]{Sta:EL1999}
Stanford, W. (1999).
\newblock On the number of pure strategy {N}ash equilibria in finite common
  payoffs games.
\newblock {\em Econom. Lett.}, 62(1):29--34.

\bibitem[Takahashi, 2008]{Tak:GEB2008}
Takahashi, S. (2008).
\newblock The number of pure {N}ash equilibria in a random game with
  nondecreasing best responses.
\newblock {\em Games Econom. Behav.}, 63(1):328--340.

\end{thebibliography}

\end{document}